\RequirePackage{fix-cm}
\documentclass[smallextended]{svjour3}       % onecolumn (second format)
\smartqed  % flush right qed marks, e.g. at end of proof
\usepackage{graphicx}
\usepackage{epstopdf}
\usepackage{dsfont,amsmath,amsfonts, amssymb, mathrsfs}
\usepackage{color}
\usepackage[colorlinks=true,linkcolor=blue]{hyperref}

\newcommand{\Cn}{ \mathds{C}^n }
\newcommand{\CO} {\Cn\setminus\Omega}
\newcommand{\COO} {\Cn\setminus\overline\Omega}
\newcommand{\Om}{\Omega = \left\{z\in\Cn : \rho(z)<0 \right\}}
\newcommand{\dom}{ {\partial\Omega} }
\newcommand{\OO}{\Omega_\eps\setminus\Omega}

\renewcommand{\O}{\Omega}
\newcommand{\eps}{\varepsilon}

\newcommand{\C}{\mathbb{C}}
\newcommand{\R}{\mathbb{R}}
\newcommand{\N}{\mathbb{N}}

\newcommand{\RE}{\mathop{\mathrm{Re}}}
\newcommand{\IM}{\mathop{\mathrm{Im}}}

\newcommand{\f}{{\mathbf{f}}}

\newcommand{\SK}{\sum\limits_{k=1}^\infty}

\newcommand{\cn}{\frac{1}{(2\pi i)^n}}
\newcommand{\dbar}{\bar\partial}
\newcommand{\intl}{\int\limits}
\newcommand{\suml}{\sum\limits}

\newcommand{\dist}[2]{ \mathop{\mathrm{dist}} (#1, #2)}

\newcommand{\supp}[1]{ \mathop{\mathrm{supp}} {#1} }
\newcommand{\BMO}{ \mathop{\mathrm{BMO}}}

\newcommand{\scp}[2]{ \left\langle #1,\ #2  \right\rangle }
\newcommand{\norm}[1]{\left\lVert #1 \right\rVert}
\newcommand{\abs}[1]{\left\lvert #1 \right\rvert}

\begin{document}

\title{Constructive description of Hardy-Sobolev spaces in strictly pseudoconvex domains with minimal smoothness.\thanks{The work is supported by Russian Science Foundation Grant 19-11-00058.}
}
%\subtitle{Do you have a subtitle?\\ If so, write it here}

\titlerunning{Hardy-Sobolev spaces}        % if too long for running head

\author{Aleksandr Rotkevich
}

%\authorrunning{Short form of author list} % if too long for running head

\institute{A. Rotkevich \at
              St. Petersburg University \\
              Tel.: +7-952-2433484\\
              \email{rotkevichas@gmail.com}
}

\date{Received: date / Accepted: date}
% The correct dates will be entered by the editor

\maketitle

\begin{abstract}
Let $\O\subset\Cn$ be a strictly pseudoconvex Runge domain with $C^2$-smooth defining function, $l\in\N,$ $p\in(1,\infty).$ We prove that the holomorphic function $f$ has derivatives of order $l$ in $H^p(\O)$ if and only if there exists a sequence on polynomials $P_n$ of degree $n$ such that 
$$\sum\limits_{k=1}^{\infty}2^{2lk}\abs{f(z)-P_{2^k}(z)}^2\in L^p(\dom).$$

% \PACS{PACS code1 \and PACS code2 \and more}
% \subclass{MSC code1 \and MSC code2 \and more}
\end{abstract}

\section{Introduction}
\label{intro}
The constructive description of the smoothness of functions in terms of polynomial approximations is a classical problem that started by D. Jackson and S.Bernstein results. In 1984 E. M. Dynkin gave a constructive characterization of holomorphic Besov spaces in simply connected domains in $\C$ with "good" boundary. We continue the research (see \cite{R13,R18_1,R18_2,R19,Sh89,Sh93}) devoted to the constructive description of spaces of functions of several complex variables. In this paper we consider Hardy-Sobolev spaces in strictly pseudoconvex domains with $C^2$-smooth defining function. 

 The main obstacle for polynomial approximations in strictly pseudoconvex domain is that polynomials are not always dense in space of holomorphic functions that are continuous up to the boundary. We restrict our consideration to strictly pseudoconvex Runge domains.
\begin{definition}
The domain $\O\subset\Cn$ is \textit{Runge domain} if for every function $f$ holomorphic in the neighbourhood of $\overline\O$ and every $\eps>0$ there exist a polynomial $P$ for which $\abs{f(z)-P(z)}<\eps,\ z\in\overline\O.$ 
\end{definition}
The condition that $\O$ is Runge domain is necessary and sufficient to obtain our results (see \cite{Sh89}).

\section{Main notations and definitions.}
\label{notations}

Let $\Cn$ be the space of $n$ complex variables, $n\geq 1,$ $z =
(z_1,\ldots, z_n),\ z_j = x_j + i y_j;$
\begin{equation*}
\partial_j f =\frac{\partial f}{\partial z_j} = \frac{1}{2}\left( \frac{\partial f}{\partial x_j} - i \frac{\partial f}{\partial y_j}\right), \quad \bar\partial_j f = \frac{\partial f}{\partial\bar{z}_j} = \frac{1}{2}\left( \frac{\partial f}{\partial x_j} + i \frac{\partial f}{\partial y_j}\right);
\end{equation*}

\[
\partial f = \suml_{k=1}^{n} \frac{\partial f}{\partial z_k} dz_k,\quad \bar{\partial} f = \suml_{k=1}^{n} \frac{\partial f}{\partial\bar{z}_k} d\bar{z}_k,\quad df=\partial f+  \bar{\partial} f;
\]

\[
\abs{\bar\partial f} = \abs{\bar\partial_1 f} + \ldots+\abs{\bar\partial_n f}
\]
We use the notation 
$
\scp{\eta}{w} = \suml_{k=1}^{n} \eta_k w_k
$
to indicate the action of the differential form $\eta=\suml_{k=1}^{n}\eta_k dz_k$ of type $(1,0)$ on the vector
$w\in\Cn.$

For a multiindex $\alpha =
(\alpha_1,\ldots,\alpha_n)\in\mathds{N}_0^n$ we set
$\abs{\alpha}~=~\alpha_1~+~\ldots~+~\alpha_n$ and
$\alpha!~=\alpha_1!\ldots\alpha_2!,$
 also $z^\alpha =
z_1^{\alpha_1}\ldots z_n^{\alpha_n}$ and $\partial^\alpha f =
\partial^{\alpha_1}_1\ldots\partial^{\alpha_n}_n f.$

Let $\Om$ be a strictly pseudoconvex domain with a $C^2$-smooth defining function $\rho:\Cn\to\R.$ We also consider a family of domains $\O_t=\left\{z\in\Cn : \rho(z)<t \right\},$ where t is small real parameter, and a $C^1$ smooth bijection $\Phi_t:\dom\to\dom_t$ given by the exponential map of a normal vector field. This allows us to define the $C^1$-smooth projection $\Psi:\O_\eps\setminus\O_{-\eps}\to\dom$ by $\Psi(\xi)=\Phi_{\rho(\xi)}^{-1}(\xi).$

For $\xi\in\dom_r = \left\{\xi\in\Cn : \rho(\xi)=r
\right\}$ we define the tangent space
\[
 T_\xi^{\R} = \left\{ z\in\Cn : \RE \scp{\partial{\rho}(\xi)}{\xi-z} = 0  \right\},
\]
the complex tangent space
\[
T_\xi = \left\{ z\in\Cn : \scp{\partial{\rho}(\xi)}{\xi-z} = 0  \right\},  
\]
the (outer) complex normal vector 
\begin{equation}\label{eq:cnormalvect}
 n(\xi)=\abs{\bar\partial\rho(\xi)}^{-1}\left(\bar\partial_1\rho(\xi),\ldots,\bar\partial_n\rho(\xi)\right).
\end{equation}
and a normal shift of $\xi$ by
\begin{equation}\label{eq:NormalShift}
 \xi_t=\xi+tn(\xi),\ t\in\R.
\end{equation}

Throughout this paper we use notations $\lesssim,\ \asymp.$  We write
$f\lesssim g$ if $f\leq c g$ for some constant $c>0,$ that doesn't
depend on main arguments of functions $f$ and $g$ and usually depend
only on the dimension $n$ and the domain $\O.$ Also $f\asymp g$ if $c^{-1}
g\leq f\leq c g$ for some $c>1.$ We denote the
Lebesgue measure in $\Cn$ as~$d\mu.$

\subsection{Area integral inequalities for Hardy spaces.}

We denote the space of holomorphic functions as $H(\O)$ and consider
the Hardy space (see~\cite{FS72,LS16,St72,S76}) 
\[
H^p(\O):=\left\{f\in
H(\O):\ \norm{f}_{H^p(\O)}=\sup\limits_{t<0} \norm{f}_{L^p(\dom_t)} <\infty\right\},
\]
and for $l\in\N$ the \textit{Hardy-Sobolev space}
\[
H_p^l(\O)=\left\{f\in H(\O): \partial^{\alpha}f\in H^p(\O),\ \abs{\alpha}=l\right\},
\]
where Lebesgue spaces $L^p(\dom_t)$ are defined by induced on the boundary of $\O_t$ Lebesgue measure $d\sigma_t.$ We also denote $d\sigma = d\sigma_0.$ 

Notice that every function $f\in H^p(\O)$ has nontangential boundary limit $f^*\in L^p(\dom)$ and $\norm{f}_{H^p(\O)}\asymp\norm{f^*}_{L^p(\dom)}.$ 

Following \cite{KL97} for $z\in\dom$ and $\eps>0$ we define the {\it Kor\'{a}nyi
region} as
$$D(z,\eta) = \{\tau\in\O\setminus\O_{-\eps} : \Psi(\tau) \in B(z,-\eta\rho(\tau))\}. $$

The strict pseudoconvexity of $\O$ and area-integral inequalities
by S.~Krantz and S.Y.~Li~\cite{KL97} imply that for $f\in H^p(\O),\
0<p<\infty,$ 
\begin{equation} \label{ineq:Luzin_internal}
  \intl_\dom d\sigma(z) \left( \sup_{\tau\in D(z,\eta)} \abs{f(\tau)}\right)^{p} \lesssim \intl_\dom \abs{f}^p d\sigma,
\end{equation}
and that
\begin{equation}\label{ineq:Lusin_internal_2}
\intl_\dom d\sigma(z) \left( \intl_0^\eps \abs{\partial f(z_{-t})}^2
 t dt\right)^{p/2} \lesssim \intl_\dom \abs{f}^p d\sigma.
\end{equation}

We point out that for every function $F$ we have
$$\intl_{\O_\eps\setminus\O} \abs{F(z)} d\mu(z) \asymp \intl_\dom d\sigma(\xi)
\intl_{0}^\eps \abs{F(\xi_t)}
dt.
$$

\section{Polynomial Leray map for strictly pseudoconve Runge domains}
 In the context of the theory of several complex variables there is no canonical reproducing formula, however we can use the Leray theorem that allows us to construct holomorphic reproducing kernels. The following lemma slightly specifies the result by Shirokov \cite{Sh89}. 
 \begin{lemma}\label{lm:Sh}
Let $\Om$ be a strictly pseudoconvex Runge domain with the $C^2$-smooth defining function $\rho.$ Then there exist functions $v(\xi,z),$  $q(\xi,z),$ $w_j(\xi,z),$ constants $\eps,\beta,\delta,s>0$ and a domain $G\subset\C$ such that
\begin{enumerate}
 \item $v(\xi,z),\ q(\xi,z),\ w_j(\xi,z)$ are polynomials in $z$ and $C^1$ in $\xi\in{\overline{\Omega}_\eps\setminus\Omega}.$
 \item $v(\xi,z)= \sum\limits_{j=1}^n w_j(\xi,z)(\xi_j-z_j)=\scp{w(\xi,z)}{\xi-z}.$
 \item $w_j(\xi,z)=\partial_j\rho(\xi)+\sum\limits_{k=1}^n P_{kj}(\xi,z)(\xi_k-z_k)$ where $P_{kj}$ are polynomials in $z.$
 \item For $\xi\in{\overline{\Omega}_\eps\setminus\Omega_{-\eps}}$ and $z\in\overline{\O_\eps},$ $\rho(z)\leq\rho(\xi),$ 
  \begin{align*}
&  \abs{v(\xi,z)}\geq \rho(\xi)-\rho(z)+\beta\abs{\xi-z}^2,\ \abs{\xi-z}<\delta; \\
&   \abs{v(\xi,z)}\geq s>0,\ \abs{\xi-z}>\delta.
  \end{align*}
 \item For every $\xi\in{\overline{\Omega}_\eps\setminus\Omega_{-\eps}}$ and $z\in\O_\eps$ with $\rho(z)\leq\rho(\xi)$ the point $\lambda(\xi,z)=v(\xi,z)q(\xi,z)$ lies in simply connected $G$ region with $C^2$-smooth boundary, which is tangent to the $y-$axis at the origin.
\end{enumerate}
\end{lemma}

\begin{proof} There exist constants (see \cite{HL84}) $\eps_1,\delta_1,\beta,s>0$ and functions $\omega_j(\xi,z),$ $M(\xi,z),$ $V(\xi,z)$ holomorphic in $z\in\O_{\eps_1}$ and $C^1$ in $\xi\in\Omega_{\eps_1}\setminus\Omega_{-\eps_1}$ with the following properties 
 \begin{align}
& \Phi(\xi,z) = \sum\limits_{j=1}^n \omega_j(\xi,z)(\xi_j-z_j);\\
& \Phi(\xi,z) \neq 0,\ z\in\overline{\O}_{\eps_1},\ \rho(\xi)>\rho(z);\\
& V(\xi,z) = \scp{\partial\rho(\xi)}{\xi-z} -\sum\limits_{k,j=1}^n  \partial^2_{ij}\rho(\xi)(\xi_j-z_j)(\xi_k-z_k);\label{eq:LevyForm}\\
& \RE{V(\xi,z)} \geq\rho(\xi)-\rho(z)+2\beta\abs{\xi-z}^2,\ \rho(\xi)\geq \rho(z),\ \abs{\xi-z}\leq\delta_1;\\
& \abs{M(\xi,z)} \geq 2s>0,\ \abs{\xi-z}<\delta_1,\ \rho(\xi)\geq\rho(z);\ M(\xi,\xi) =1;\\
& \Phi(\xi,z)=M(\xi,z)V(\xi,z),\ \abs{\xi-z}\leq\delta_1,\ \rho(\xi)\geq\rho(z);
 \end{align}
Note that for $\abs{\xi-z}\leq\delta_1,$ $\rho(\xi)\geq\rho(z),$
\begin{multline*}
 \sum\limits_{j=1}^n \omega_j(\xi,\xi)(\xi_j-z_j) = \sum\limits_{j=1}^n \omega_j(\xi,z)(\xi_j-z_j)\\ + \sum\limits_{j=1}^n (\omega_j(\xi,\xi)-\omega_j(\xi,z))(\xi_j-z_j) 
 = M(\xi,z)\scp{\partial\rho(\xi)}{\xi-z} + O(\abs{\xi-z}^2)\\ = M(\xi,\xi)\scp{\partial\rho(\xi)}{\xi-z}+ O(\abs{\xi-z}^2).
\end{multline*}
and $\omega_j(\xi,\xi)=\partial_j\rho(\xi).$

By the Oka-Hefer lemma (see \cite{HL84}) for some $\eps_2>0$ there exist functions $\gamma_{kj}(\zeta,\xi,z)$ holomorphic in $\xi,z\in\O_{\eps_2}$ and $C^1$-smooth in $\zeta\in\O_{\eps_2}\setminus\overline{\O}_{-\eps_2}$ such that
\begin{equation}
 \omega_j(\zeta,z)-\omega_j(\zeta,\xi) = \sum\limits_{k=1}^n \gamma_{kj}(\zeta,\xi,z)(z_k-\xi_k).
\end{equation}

We want to approximate $\gamma_{kj}(\xi,\xi,z)$ by polynomials so that the degree does not depend on $\xi.$ Let $\eta>0.$ Since $\gamma_{kj}(\xi,\xi,z)$ is $C^1$-smooth in $\xi$ we can choose $\lambda>0$ such that 
$$\max\limits_{k,j}\abs{\gamma_{kj}(\xi,\xi,z)-\gamma_{kj}(\hat{\xi},\hat{\xi},z)}<\eta,\ \abs{\xi-\hat{\xi}}<\lambda,\ z\in\O_{\eps_2},\ \xi,\hat\xi\in\O_{\eps_2}\setminus\overline{\O}_{-\eps_2}.$$

Let $E=\{\xi_\mu\}_{\mu=1}^N$ be a finite $\lambda/2-$net of a set $\overline{\O}_{\eps_2}\setminus\O_{-\eps_2}.$ Then 
$$\overline{\O}_{\eps_2}\setminus\O_{-\eps_2}\subset\bigcup\limits_{\mu=1}^N B_\mu,\ B_{\mu}=B(\xi_\mu,\lambda).$$

Since $\O_{\eps_2}$ is Runge we have polynomials $P_{kj}^\mu$ such that
\begin{equation}\label{eq:P_kj^mu}
 \abs{\gamma_{kj}(\xi_\mu,\xi_\mu,z)- P_{kj}^\mu(z)}<\eta,\quad z\in\O_{\eps_2}.
\end{equation} 
Then
$$\abs{\gamma_{kj}(\xi,\xi,z)-P_{kj}^\mu(z)}< 2\eta,\ \xi\in B(\xi_\mu,\lambda),\ z\in\O_{\eps_2}.$$

We consider a smooth partition of unity subordinated to covering
$\{B_\mu\}$  
$$\chi_\mu \in C^{\infty}(\C^n),\ 0\leq\chi_\mu\leq 1,\ \supp{\chi_\mu}\subset B_\mu,\ \suml_{\mu=1}^N \chi_\mu(\xi) = 1,\ \xi\in\O_{\eps_2}\setminus\overline{\O}_{-\eps_2}$$
and let 
\begin{equation}\label{eq:P_kj}
P_{kj}(\xi,z)=\suml_{\mu=1}^N \chi_\mu(\xi)P_{kj}^\mu(z),
\end{equation}
which implies
\begin{equation}\label{eq:P_kj_est}
 \abs{\gamma_{kj}(\xi,\xi,z)-P_{kj}(\xi,z)}\leq \suml_{\mu: \abs{\xi-\xi_\mu}<\lambda} \chi_\mu(\xi)\abs{\gamma_{kj}(\xi,z)-P_{kj}^\mu(z)}\leq 2\eta.
\end{equation}

Let 
\begin{equation}\label{eq:w_j}
w_j(\xi,z) = \omega_j(\xi,\xi)(\xi_j-z_j) +  \sum\limits_{k=1}^nP_{kj}(\xi,z)(\xi_k-z_k),
\end{equation}
\begin{equation}\label{eq:v}
v(\xi,z)= \scp{w_j(\xi,z)}{\xi-z}.
\end{equation}
Then
$$\abs{\Phi(\xi,z)-v(\xi,z)}\leq c \eta\abs{\xi-z}^2$$
and choosing $\eta>0$ small enough we obtain functions $v(\xi,z)$ and $w_j(\xi,z)$ satisfying the first four conditions of lemma. 

Now we pass to a construction of function $q(\xi,z).$ We see now that
for some $a,R>0$ 
$$\lambda=v(\xi,z)\in \{\abs{\lambda}\leq R,\ \abs{\lambda+a}\geq a\}$$
 for every $\xi\in\overline{\O}_\eps\setminus\O,\ z\in\overline{\O}.$ Hence 
$$\frac{ v(\xi,z) }{ v(\xi,z)+a }\in B_0=\{ \tau : \abs{\tau-1}<1 \}.$$
We approximate $h(\xi,z)=1/(a+v(\xi,z))$ by a polynomial $q$ analogously to approximation of $\Phi(\xi,z).$ By Oka-Hefer lemma we have
$$h(\zeta,z)-h(\zeta,\xi) = \sum\limits_{j=1}^n (z_j-\xi_j)H_j(\zeta,\xi,z),\quad \xi,
\zeta\in\overline{\Omega}_{\eps_1},\ z\in\overline{\Omega}_{\eps_1}\setminus\O,$$
and as in (\ref{eq:P_kj^mu}-\ref{eq:v}) we construct polynomials $q_j(\xi,z)$ such that
$$\abs{H_j(\xi,\xi,z)-q_j(\xi,z)}<\eta.$$
Let $q(\xi,z)=h(\xi,\xi)+\sum\limits_{j=1}^n(z_j-\xi_j)q_j(\xi,z).$ Then $$\abs{\frac{v(\xi,z)}{a+v(\xi,z)}-v(\xi,z)q(\xi,z)}\leq c\eta\abs{v(\xi,z)}\abs{\xi-z}\leq c\eta\abs{\xi-z}^2,$$
for $z\in\overline{\O},\ \xi\in\overline{\O}_{\eps}\setminus\O.$

Now we choose $\eta>0$ small enough to satisfy the last condition.\qed
\end{proof}

\begin{remark} We note that
\begin{equation}
 \rho(\xi)-\rho(z)+\abs{\xi-z}^2\lesssim \abs{v(\xi,z)}\lesssim \abs{\xi-z},\quad \xi,z\in\O_\eps\setminus\O_{-\eps},\ \rho(\xi)\geq\rho(z).
\end{equation}

\end{remark}

In the following three lemmas we will adapt ideas by L. Lanzani and E.M. Stein to study the function 
$$d(\xi,z)=\abs{v(\xi,z)}=\abs{\scp{w(\xi,z)}{\xi-z}}$$
and to prove that $d$ defines on $\dom_t,$ $\abs{t}<\eps,$ a quasimetric.

\begin{lemma} \label{lm:QM1} Let $\xi,z\in\Omega_{\eps}\setminus\O_{-\eps},$ $\rho(z)\leq\rho(\xi)\leq\rho(\zeta).$ Then
\begin{enumerate}
\item $\abs{v(\xi,z)-v(\zeta,z)}\lesssim d(\zeta,\xi)+d(\zeta,\xi)^{1/2}d(\xi,z)^{1/2}.$
\item $d(\xi,z)\asymp \rho(\xi)-\rho(z) + \abs{\IM\scp{\partial\rho(\xi)}{\xi-z}} +\abs{\xi-z}^2.$ 
\end{enumerate}
\end{lemma}

\begin{proof}
 1. We begin by noting that
 \begin{equation*}
 v(\xi,z)=\scp{\partial\rho(\xi)}{\xi-z}+Q(\xi,z),
\end{equation*}
where $\abs{Q(\xi,z)}\lesssim\abs{\xi-z}^2,$ and the estimate of $\abs{v(\xi,z)-v(\zeta,z)}$ may be obtained in two steps. First
 \begin{multline*}
 \abs{\scp{\partial\rho(\xi)}{\xi-z}-\scp{\partial\rho(\zeta)}{ \zeta-z } }\leq \abs{\scp{\partial\rho(\zeta)}{ \zeta-\xi }}+\\
 \abs{\scp{\partial\rho(\zeta)-\partial\rho(\xi)}{\xi-z} }\lesssim 
 d(\zeta,\xi)+\abs{\zeta-\xi}\abs{\xi-z}\\
 \lesssim d(\zeta,\xi)+d(\zeta,\xi)^{1/2}d(\xi,z)^{1/2}.
 \end{multline*}
 Next, we have 
 \begin{multline*}
 \abs{Q(\xi,z)-Q(\zeta,z)} \leq \sum\limits_{k,j=1}^n \abs{P_{kj}(\xi,z)-P_{kj}(\zeta,z)}\abs{(\xi_k-z_k)(\xi_j-z_j)}\\ + \sum\limits_{k,j=1}^n \abs{P_{kj}(\zeta,z)}\abs{(\xi_k-z_k)(\xi_j-z_j)-(\zeta_k-z_k)(\zeta_j-z_j)}\\ \lesssim\abs{\xi-\zeta}\abs{\xi-z}^2 + \abs{\xi-\zeta}\abs{\xi-z} + \abs{\xi-\zeta}^2 \lesssim d(\zeta,\xi)+d(\zeta,\xi)^{1/2}d(\xi,z)^{1/2}.
 \end{multline*}
 
 2. Let $\xi,z\in\overline{\O}_\eps\setminus\O_{-\eps},\ \rho(\xi)\geq\rho(z).$ We decompose $z$ as $z=u+tn(\xi),$ where $u\in T_\xi,\ t\in\C$ and $n(\xi)$ is a complex normal vector at $\xi.$ Then 
 $$\rho(z)=\rho(\xi)+2\RE\scp{\partial\rho(\xi)}{z-\xi}+O(\abs{\xi-z}^2)=\rho(\xi)+2\abs{\partial\rho}\RE{t}+O(\abs{\xi-z}^2)$$
 and
 $$2\RE{v(\xi,z)}=\rho(\xi)-2\abs{\partial\rho(\xi)}\RE{t}+O(\abs{\xi-z}^2)=\rho(\xi)-\rho(z)+O(\abs{\xi-z}^2).$$
Second,
 $$2\IM{v(\xi,z)}=-\abs{\partial\rho(\xi)}\IM{t}+O(\abs{\xi-z}^2).$$
and
\begin{align*}
\abs{\IM{v(\xi,z)}} & \lesssim\abs{\IM{t}}+\abs{\xi-z}^2;\\
\abs{\IM{t}} & \lesssim \abs{\IM{v(\xi,z)}}+\abs{\xi-z}^2.
\end{align*}
This implies that 
$$\abs{v(\xi,z)}\lesssim \rho(\xi) -\rho(z)+\abs{\IM{t}}+\abs{\xi-z}^2. $$

To obtain the lower estimate we note that
$$\abs{v(\xi,z)}\geq c(\rho(\xi) -\rho(z)+\abs{\xi-z}^2),$$
$$ \abs{v(\xi,z)}\geq c'\abs{\IM{t}}-\abs{\xi-z}^2,$$
and
$$(1+c/2)\abs{v(\xi,z)} \geq c(\rho(\xi) -\rho(z))+\frac{cc'}{2}\abs{\IM{t}}+\frac{c}{2}\abs{\xi-z}^2$$
which finishes the proof of the lemma.
\end{proof}

\begin{lemma}\label{lm:QM2} %For $\zeta\in\O_\eps, t\in\R,$ we denote $\zeta_t=\zeta+t n(\zeta).$ 
\begin{enumerate}
\item Let $\xi\in\dom,\ z\in\O,\ t\in [0,\eps].$ Then $d(\xi_t,z)\asymp t+d(\xi,z).$
\item  Let $\xi\in\O_\eps\setminus\O,z\in\dom,\ t\in[-\eps,0].$ Then
$d(\xi,z_t)\asymp -t+d(\xi,z).$
\item Let $\xi\in\O_\eps\setminus\O,\ z\in\overline{\O}.$ Then
$d(\xi,z)\asymp \rho(\xi)-\rho(z)+d(\Psi(\xi),\Psi(z)).$
\end{enumerate}
\end{lemma}
\begin{proof}
 1. Note that $\rho(\xi_t)\asymp t$ and
 \begin{multline*}
 \IM\scp{\partial\rho(\xi_t)}{\xi_t-z} = \IM\scp{\partial\rho(\xi)}{\xi+tn(\xi)-z} + O(t\abs{\xi_t-z})\\ = \IM\scp{\partial\rho(\xi)}{\xi-z} + O(t^2+t\abs{\xi-z}).
 \end{multline*}
 Then
 \begin{multline*}
  d(\xi_t,z)\asymp t-\rho(z)+\abs{\IM\scp{\partial\rho(\xi)}{\xi-z}}+\abs{\xi_t-z}^2 +O(t^2+t\abs{\xi-z})\asymp\\ t-\rho(z)+\abs{\IM\scp{\partial\rho(\xi)}{\xi-z}}+\abs{\xi-z}^2 +t^2\asymp t+d(\xi,z)
 \end{multline*}
2. If $t<0,$ then $z_t\in\O$ and
\begin{multline*}
\IM\scp{\partial\rho(\xi)}{\xi-z_t}= \IM\scp{\partial\rho(\xi)}{\xi-z} - 
t\IM\scp{\partial\rho(z)}{n(z)} \\- t\IM\scp{\partial\rho(\xi)-\partial\rho(z)}{n(z)} = -t\abs{\partial\rho(z)}+\IM\scp{\partial\rho(\xi)}{\xi-z} + O(t\abs{\xi-z}).
\end{multline*}
Consequently,
\begin{multline*}
d(\xi,z_t)\asymp -t+ \abs{\IM\scp{\partial\rho(\xi)}{\xi-z_t}} +\abs{\xi-z_t}^2 \\
= - t+ \abs{\IM\scp{\partial\rho(\xi)}{\xi-z_t}} +\abs{\xi-z}^2 +t\scp{\xi-z}{n(z)}+t^2  \\
\asymp -t+\abs{\IM\scp{\partial\rho(\xi)}{\xi-z}}+\abs{\xi-z}^2\asymp -t+ d(\xi,z).
\end{multline*}
3. Finally $\xi=\Psi(\xi)+tn(\Psi(\xi))\in\O_\eps\setminus\O$ with $t\asymp\rho(\xi),$ and\\ $z=\Psi(z)+sn(\Psi(z))\in\overline{\O}$ with $-s\asymp-\rho(z).$ Then
$$d(\xi,z)\asymp t+d(\Psi(\xi),z)\asymp t+s+d(\Psi(\xi),\Psi(z))\asymp \rho(\xi)-\rho(z)+d(\Psi(\xi),\Psi(z))$$
and this finishes the proof of the lemma.\qed
\end{proof}

The next Lemma shows that function $d$ defines on $\dom_t$ a quasimetric.
\begin{lemma}\label{lm:QM3} There exist a constant $A>0$ such that for  $z,\zeta,\xi\in\dom_t,$ $\abs{t}<\eps.$
\begin{equation}\label{eq:QM}
d(\xi,z)\leq Ad(z,\xi),\ d(z,\zeta)\leq A(d(z,\xi)+d(\xi,\zeta)).
\end{equation}
\end{lemma}
\begin{proof}
 Let $\xi,z,\zeta\in\dom_t.$ Then
 $$\IM{\scp{\partial\rho(\xi)}{\xi-z}} = -\IM{\scp{\partial\rho(z)}{z-\xi}} + 
 \IM{\scp{\partial\rho(z)-\rho(\xi)}{z-\xi}} $$
 and
 $$d(\xi,z)\lesssim \abs{\IM{\scp{\partial\rho(z)}{z-\xi}}} +\abs{z-\xi}^2\lesssim d(z,\xi).$$
 To prove the triangle inequality consider
 \begin{multline*}
 d(z,\zeta)\lesssim \abs{\IM\scp{\partial\rho(z)}{z-\zeta}} +\abs{z-\zeta}^2\\
  \leq  \abs{\IM\scp{\partial\rho(\xi)}{z-\zeta}} + \abs{\IM\scp{\partial\rho(z)-\partial\rho(\xi)}{z-\zeta}}+\abs{z-\xi}^2+\abs{\xi-\zeta}^2 \\
 \lesssim  \abs{\IM\scp{\partial\rho(\xi)}{\xi-\zeta}}  + \abs{\IM\scp{\partial\rho(\xi)}{\xi-z}} +\abs{z-\xi}\abs{z-\zeta} + \abs{z-\xi}^2+\abs{\xi-\zeta}^2\\ 
 \lesssim \abs{\IM\scp{\partial\rho(\xi)}{\xi-\zeta}}  + \abs{\IM\scp{\partial\rho(\xi)}{\xi-z}} +\abs{z-\xi}^2+  \abs{\xi-\zeta}^2\\ 
 \lesssim d(\xi,z)+d(\xi,\zeta)\lesssim d(z,\xi)+d(\xi,\zeta).
 \end{multline*}
 This finishes the proof of the lemma.\qed
\end{proof}

\section{Leray-Koppelman formula and pseudoanalytic continuation}

By Leray-Koppelman formula for every $f\in H^1(\O)$  we have
\begin{equation} \label{eq:CLF}
 f(z) = K f(z) =  \int\limits_{\dom}  \frac{f^*(\xi) \omega(\xi,z)}{\scp{w(\xi,z)}{\xi-z}^n},\ z\in\Omega,
\end{equation}
where $\omega(\xi,z) =\cn w(\xi,z)\wedge \left(\bar\partial_\xi w(\xi,z)\right)^{n-1}.$

Let $\f\in C^1(\O_\eps\setminus\overline\O)$ with $\supp\f\subset\O_\eps$ and assume that nontangential boundary values of $f$ and $\f$ coincide on $\dom.$ Then applying Stoke's theorem to Leray-Koppelman integral and pseudoanalytic continuation $\f$ we have (see \cite{R18_1} for details)
\begin{equation} \label{LK_co}
f(z) = \int\limits_{\O_\eps\setminus\O} \frac{ \bar\partial \f(\xi)\wedge \omega(\xi,z)}{\scp{w(\xi,z)}{\xi-z}^n},\ z\in\O,
\end{equation}
since
\[
d_\xi\frac{ \omega(\xi,z)}{\scp{w(\xi,z)}{\xi-z}^n}  = 0,\quad z\in\O,\ \xi\in\OO.
\]
We denote the kernel by $K(\xi,z)=\frac{\omega(\xi,z)}{\scp{w(\xi,z)}{\xi-z}^{n}}.$
\begin{definition} The function $\f\in C^1(\O_\eps\setminus\overline\O)$ with $\supp\f\subset\O_\eps$ is a \textit{pseudoanalytic continuation} of function $f\in H(\O)$ if the identity (\ref{LK_co}) holds.
\end{definition}
%\section{Estimates of the reproducing kernel} \label{Est}

If $B(z,\delta)=\{\xi\in\dom:d(\xi,z)<\delta\}$ is a quasiball with respect to $d$ then $\sigma(B(z,\delta))\lesssim \delta^n.$ Thus $\mu(V(\xi,\delta))\lesssim\delta^{n+1},$ where $V(\xi,\delta)=\{z\in\OO : d(\xi,z)<\delta\}$ and analogously to \cite{LS13,R13} we have the following classical estimates. 
\begin{lemma}  \label{LerayEst}
 Let $\alpha>0$ and $0<r<\delta<\eps.$ Then
% I_\alpha(z,\delta) = J_\alpha(z,\delta) = 
\begin{align*}
& \int\limits_{z\in\dom,\ d(\xi,z)>\delta} \frac{d\sigma(z)}{d(\xi,z)^{n+\alpha}} \lesssim \delta^{-\alpha},\quad \xi\in\OO;\\
& \int\limits_{\xi\in\dom_r,\ d(\xi,z)>\delta}
\frac{d\sigma_r(\xi)}{d(\xi,z)^{n+\alpha}} \lesssim
\delta^{-\alpha},\quad z\in\Omega;\\
& \int\limits_{z\in\dom,\ d(\xi,z)<\delta} \frac{d\sigma(z)}{d(\xi,z)^n} \lesssim 1+ \log{\frac{\delta}{r}},\quad \rho(\xi) = r<\eps;\\
& \int\limits_{\xi\in\dom_r,\ d(\xi,z)<\delta} \frac{d\sigma_r(\xi)}{d(\xi,z)^n} \lesssim 1+ \log{\frac{\delta}{r}},\quad z\in\Omega.
\end{align*}

\end{lemma}

Concluding ideas by N. Shirokov we have the following polynomial approximation of a kernel.

\begin{lemma}[Shirokov \cite{Sh89}] \label{lm:KL_approx}
Let $\O$ be a strictly pseudoconvex Runge domain and $r>0.$ Then for every
$m\in\N$ there exists a function $K^{glob}_m(\xi,z)$ which is continuous in $\xi\in\OO,$  polynomial in $z$ with $\deg K^{glob}_m(\xi,\cdot)\lesssim m$ and satisfies the
following properties:

\begin{equation*} %\label{eq:CLF_approx1}
 \abs{K(\xi,z) - K^{glob}_m(\xi,z)} \lesssim \frac{1}{m^{r}} \frac{1}{d(\xi,z)^{n+r}},\quad d(\xi,z)\geq \frac{1}{m};
\end{equation*}

\begin{equation*} %\label{eq:CLF_approx2}
 \abs{K^{glob}_m(\xi,z)}\lesssim m^n,\quad d(\xi,z)\leq\frac{1}{m}.
\end{equation*}

\end{lemma}

\section{Two methods of pseudoanalytic continuation}
Let $\f$ be a pseudoanalytic continuation of function $f\in H^1(\O)$ and $1\leq p\leq\infty.$ We introduce the following important characteristics of function $\f$
\begin{equation} \label{eq:SP}
 S_p(\f,r) = \norm{ \bar{\partial} \f }_{L^p(\dom_r)},\ r>0.
\end{equation}
In this section we generalize ideas by E.M. Dynkin \cite{D81} to construct pseudoanalytic continuations with some estimates in this value.

\subsection{Continuation by symmetry}

For $z\in\O_\eps\setminus\O$ we define the point $z^*$ symmetric to $z$ with respect to $\partial\O$ by 
\begin{equation}\label{eq:sym}
z^*=\Psi(z)-\dist{z}{\dom}n(\Psi(z)).
\end{equation}

\begin{theorem} \label{thm:PAC_sym}
 Let $f\in H^1_p(\O)$ and $1<p<\infty,\ m\in\N.$ Then there exist a pseudoanalytic continuation
 $\f\in C^1(\COO)$ of the function $f$ such that
 $\supp{\f}\subset\O_\eps,$
 $\abs{\dbar \f}\in L^p(\O_\eps\setminus\O)$ and
 \begin{equation} \label{ineq:PAC_sym0}
  \abs{\dbar \f(z)} \lesssim \max\limits_{\abs{\alpha}=m}
  \abs{\partial^{\alpha}f(z^*)}
  \rho(z)^{m-1},\quad z\in\O_\eps\setminus\O.
 \end{equation}
\end{theorem}

\begin{proof}
Define
\begin{equation} \label{ineq:PAC_sym1}
 \f_0(z) = \sum\limits_{\abs{\alpha}\leq m-1} \partial^{\alpha}f(z^*)
 \frac{(z-z^*)^\alpha}{\alpha !},\ z\in\OO.
\end{equation}
Let $\alpha\pm e_k=(\alpha_1,\ldots,\alpha_k\pm 1,\alpha_n)$ if $\alpha_k\neq0$ and
$(z-z^*)^{\alpha-e_k}=0$ if $\alpha_k=0.$ With these notations
we have
\begin{multline}
 \dbar_j \f_0=\SK \sum\limits_{\abs{\alpha}\leq m-1} \left(\partial^{\alpha+e_k}f(z^*)
 \frac{(z-z^*)^\alpha}{\alpha !} - \partial^{\alpha}f(z^*)
 \frac{(z-z^*)^{\alpha-e_k}}{(\alpha-e_k)!}\right)\dbar_j z^*_k\\
 = \SK \sum\limits_{\abs{\alpha}=m-1} \partial^{\alpha+e_k}f(z^*)
 \frac{(z-z^*)^\alpha}{\alpha !} \dbar_j z^*_k,
\end{multline}
hence, $$ \abs{\dbar \f_0(z)} \lesssim \max\limits_{\abs{\alpha}=m}
\abs{\partial^{\alpha}f(z^*)}
  \rho(z)^{m-1},\quad z\in\CO. $$
Consider a function $\chi\in C^\infty(0,\infty)$ such that $\chi(t)=1$
for $t\leq \eps/2$ and $\chi(t)=0$ for $t\geq \eps.$ The function
$\f(z) = \f_0(z)\chi(\rho(z))$ satisfies condition
(\ref{ineq:PAC_sym0}) and $\supp\f~\subset~\O_\eps.$

Let $d = \dist{z^*}{\dom}/10$ and $D(z)= D(\Psi(z),c_0d).$ Then by the Cauchy maximal inequality for every multiindex $\alpha$
such that $\abs{\alpha}=m$  we have
\begin{multline*}
\abs{\partial^\alpha f(z^*)} \lesssim d^{-m+1} \sup\{
\abs{\partial f(\tau)}:\abs{\tau-z^*}<d\}\\ \lesssim \rho(z)^{-m+1} \sup\{ \abs{\partial f(\tau)}:\tau\in
D(z)\},
\end{multline*}
 for some $c_0>0.$
 
Finally, by (\ref{ineq:Luzin_internal}) we get
\begin{equation*}
\int_{\OO} \abs{\dbar \f}^p d\mu \lesssim \int_{\O\setminus\O_{-\eps}} \left(\sup\limits_{\tau\in D(z)} \abs{\partial f(\tau)}\right)^pd\mu(z) 
\lesssim \norm{\partial f}^p_{H^p(\O)}<\infty.
\end{equation*}
Thus $\abs{\dbar \f}\in L^p(\O_\eps\setminus\O)$ and this finishes the proof of the theorem. \qed
\end{proof}

\subsection{Pseudoanalytic continuation by global polynomial approximations.\label{Cont_glob}}
Recently L. Lanzani and E.M. Stein in \cite{LS16} proved that strict pseudoconvexity of domain $\O$ implies that functions holomorphic in neighbourhood of $\O$ are dense in $H^p(\Omega)$ with $1<p<\infty$ even if the defining function is $C^2-$smooth. Also every holomorphic in neighbourhood of $\O$ function can be approximated on $\overline{\O}$ by polynomials since $\O$ is Runge. Thus there exists a polynomial sequence $P_1,P_2,\ldots$ converging to $f^*$ in $L^p(\dom).$ Let 
\begin{equation}
\lambda(z) = \rho(z)^{-1} \abs{P_{2^{m+1}}(z) - P_{2^{m}}(z)},\quad 2^{-m}<\rho(z)\leq 2^{-m+1}.
\end{equation}

\begin{theorem} \label{thm:PAC_glob}
 Assume that $\lambda\in L^p(\OO)$ for some $p\geq 1.$ Then there exist a pseudoanalytic continuation $\f$ of the function $f$
 such that
\begin{equation}\label{eq:PAC_la}
\abs{\bar{\partial} \mathbf{f}(z)} \lesssim \lambda(z),\quad
z\in\O_\eps\setminus\overline{\O}.
\end{equation}
\end{theorem}

\begin{proof}
Consider a function $\chi\in C^\infty(0,\infty)$ such that $\chi(t)=1$
for $t\leq \frac{5}{4}$ and $\chi(t)=0$ for $t\geq \frac{7}{4}.$ We let 
for $m\in\N$ 
\[
\mathbf{f}_0(z) =  P_{2^{m}}(z) + \chi(2^m\rho(z)) (P_{2^{m+1}}(z) - P_{2^{m}}(z)),\ 2^{-m}<\rho(z)<2^{-m+1},
\]
and define the continuation of the function $f$ by formula $\f=\chi(2\rho(z)/\eps)\f_0(z).$

Now $\mathbf{f}$ is $C^1$-function on $\Cn\setminus\overline{\Omega}$ and
$\abs{\bar{\partial} \mathbf{f}(z)} \lesssim \lambda(z).$ We define
a function $F_m(z)$ as $F_m(z)= \mathbf{f}(z)$ for $\rho(z)>2^{-m}$
and as $F_m(z) = P_{2^{m+1}}(z)$ for $\rho(z)<2^{-m}.$ The function $F_m$ is smooth and holomorphic in $\Omega_{2^{-m}},$ and
$\abs{\dbar F_m(z)}\lesssim\lambda(z)$ for
$z\in\Cn\setminus\Omega_{2^{-m}}.$ Thus %similarly to (\ref{eq:PAC0})
 we get
$$P_{2^{m+1}}(z) = F_m(z) = \int\limits_{\O_\eps\setminus\O} \frac{\bar{\partial}F_m(\xi) \wedge\omega(\xi,z)}{\scp{w(\xi,z)}{\xi-z}^n},\ z\in\Omega,$$

We can pass to the limit in this formula by the dominated
convergence theorem. Hence, the function~$\f$ is a pseudoanalytic continuation of the
function~$f$. \qed
\end{proof}

%This following lemma allows us to connect 
\begin{lemma}\label{lm:Pm}
Let $P_{2^m}$ be a polynomial of degree $2^m$ and $1\leq p\leq\infty.$ Then
\begin{equation}
\norm{P_{2^m}}_{L^p(\dom_r)}\lesssim \norm{P_{2^m}}_{L^p(\dom)},\ 2^{-m}\leq r\leq 2^{-m+1}
\end{equation}
and the constant does not depend on $m$.
\end{lemma}
\begin{proof}
Let $\xi\in\dom$ and $n=n(\xi)$ be a complex normal at this point. For some $\delta>0$ and every $w\in T_\xi$ such that $\abs{\xi-w}<\delta$ a set $$\gamma_w=\{u\in\C: w+u n\in\dom,\ \abs{u}<\delta \}$$ is a simple non closed $C^2$-smooth curve. Also curves $\tilde{\gamma}_w=\{w+un: u\in\gamma_w,\ |u|<\delta\}$ cover a neighbourhood of $\xi$ in $\dom.$

%can be continued to a conformal map of closed domains and

Consider a conformal map 
$$\psi_w:\C\setminus\gamma_w\to\{|v|>1\},\ \psi_w'(\infty)>0.$$
 The smoothness of $\dom$ implies that there exist a neighbourhood $V=V_\xi$ of point $\xi$ in $\Cn$ and constants $c_1,c_2>0$ that do not depend on $\xi$ such that 
\begin{equation*}%\label{eq}
\gamma_{w,r}=\{u\in\C:w+un\in\dom_r\cap V\}\subset\{u\in\C: c_1r<\psi_w(u)-1<c_2r\}.
\end{equation*}
Thus
\[
\abs{\psi_w(u)}^{2^{m+2}}\asymp (1+2^{-m})^{2^m}\asymp 1,\ 2^{-m}\leq \rho(w+un)< 2^{-m+1};
\]
\[
\abs{\psi_w(u)}=1,\ u\in\gamma_w.
\]
Consider a function
\begin{equation*}
H_w(u)=P_{2^m}(w+un)\psi_w(u)^{-2^{m+2}},
\end{equation*}
holomorphic in $\C\setminus\gamma_w$ such that $\abs{u}H_w(u) \to 0,\ u\to\infty.$ Then
\begin{equation*}
\sup\limits_{r>0}\norm{H_w}_{L^p(\abs{\psi_w}=1+r)}
\lesssim\norm{H_w}_{L^p(\gamma_w)}
\end{equation*}
and
\begin{equation*}
\sup\limits_{r>0}\norm{H_w}_{L^p(\gamma_{w,r})}
\lesssim\norm{H_w}_{L^p(\gamma_w)}.
\end{equation*}
Hence,  
\begin{equation*}
 \norm{P_{2^m}}_{L^p(\gamma_{w,r})}\lesssim \norm{H_w}_{L^p(\gamma_{w,r})}\lesssim \norm{H_w}_{L^p(\gamma_{w})} =\norm{P_{2^m}}_{L^p(\gamma_{w})}
\end{equation*}
for $2^{-m }\leq r< 2^{-m+1}$
and integrating this estimate by $w\in T_\xi$ we get
\begin{equation*}
 \norm{P_{2^m}}_{L^p(\dom_r\bigcap V)}\lesssim \norm{P_{2^m}}_{L^p(\dom\bigcap V)},\ 2^{-m}\leq r<2^{-m+1}.
\end{equation*}

Finally, we choose the finite covering of $\dom$ that also covers a set $\O_{2^{-m+1}}\setminus\O_{2^{-m}}$ and obtain the desired estimate.\qed

\end{proof}
\begin{corollary}\label{cor:ContGlob}
The continuation $\f$ in Theorem~\ref{thm:PAC_glob} satisfies an estimate
\begin{equation} \label{eq:SpEn1}
S_p(\f,r) \lesssim 2^m E_{2^m}(f)_p,\  2^{-m}\leq r\leq 2^{-m+1}.
\end{equation}
\end{corollary}

\subsection{Pseudoanalytic continuation of Hardy-Sobolev spaces\label{PAC_Sobolev}}
\begin{theorem} \label{thm:PAC_Sobolev}
Let $\O$ be a strictly pseudoconvex domain, $1<p<\infty,$ $l\in\N$ and
$f\in H^1(\O).$ Then $f\in H_p^l(\O)$ if and only if there exists
such pseudoanalytic continuation $\f$ that for some $\eps, \eta>0$

\begin{equation} \label{ineq:PAC_Sobolev}
 \intl_\dom d\sigma(z) \left(\intl_0^\eps \abs{\dbar \f(z_t) t^{-l}}^2 t dt
 \right)^{p/2} <\infty.
\end{equation}
\end{theorem}

\begin{proof}
Let $f\in H^l_p(\O),$ then $\partial^\beta f\in H^p(\O)$ for $\abs{\beta}=l.$ By Theorem~\ref{thm:PAC_sym} we can
construct a pseudoanalytic continuation~$\f$ such that
$$\abs{\dbar \f(z)} \lesssim \max\limits_{\abs{\alpha}=l+1} \abs{\partial^\alpha
f(z^*)}  \rho(z)^{l},\quad z\in\CO.$$

Applying the estimate~(\ref{ineq:Lusin_internal_2}) we
obtain
\begin{multline*}
\intl_\dom d\sigma(z) \left(\intl_0^\eps \abs{\dbar \f(z_t)t^{-l}}^2 tdt  \right)^{p/2}\\
 \lesssim \max\limits_{\abs{\alpha}=l+1} \intl_\dom d\sigma(z) \left(\intl_0^\eps \abs{\partial^{\alpha}f(z_{-t})}^2tdt
 \right)^{p/2} <\infty.
\end{multline*}

To prove the sufficiency, assume that the function $f\in H^1(\O)$ admits
the pseudoanalytic continuation~$\f$ with  estimate~(\ref{ineq:PAC_Sobolev}). We will prove that for every function
$g\in L^{p'}(\dom),\ \frac{1}{p}+\frac{1}{p'} =1,$ and every
multiindex $\alpha,\ \abs{\alpha}\leq l,$

$$ \sup_{-\eps<s<0}\abs{ \intl_{\dom_{s}} g(\Psi(z)) \partial^{\alpha}f(z) dS(z)} \leq c(f)\norm{g }_{L^{p'}(\dom)},$$
where $dS=\partial{\rho}\wedge(\bar\partial\partial{\rho})^{n-1}$ is a Leray-Levy measure. We point out that measure $dS$ is equivalent to a Lebesgue surface measure $d\sigma_s$ since $\O_s$ is strictly pseudoconvex.

%Assume, without loss of generality, that $\alpha=(l,0,\ldots,0).$ 
By representation~(\ref{LK_co}) we have
$$ f(z) = \intl_{\CO} \frac{\bar{\partial}\f(\xi) \wedge \omega(\xi,z)}{v(\xi,z)^n}$$
and
\begin{equation*}
\partial^\alpha f(z) = \intl_{\CO} \bar{\partial}\f(\xi) \wedge \partial^\alpha_z\frac{\omega(\xi,z)}{v(\xi,z)^n}=\intl_{\CO} \lambda(\xi,z)\frac{\bar{\partial}\f(\xi) \wedge \omega(\xi,z)}{v(\xi,z)^{n+l}},\ z\in\O,
\end{equation*}
where $\lambda(\xi,z)$ is a polynomial in $z$ and $C^1$ function in $\xi\in\O_\eps\setminus\O.$

Then
\begin{multline*}
 \intl_{\dom_s} g(\Psi(z))\partial^{\alpha}f(z)  dS(z)
 = \intl_{\dom_s} g(\Psi(z)) \left( \intl_{\CO} \lambda(\xi,z) \frac{\bar{\partial}\f(\xi) \wedge \omega(\xi,z)}{v(\xi,z)^{n+l}}\right)
 dS(z)\\
 = \intl_{\CO} \ \bar{\partial}\f(\xi) \wedge
 \omega(\xi,z) \intl_{\dom_s}
 \frac{g(\Psi(z))\lambda(\xi,z)}{v(\xi,z)^{n+l}}dS(z).
\end{multline*}
Define $G_l(\xi) = \intl_\dom
 \frac{g(\Psi(z))\lambda(\xi,z)}{v(\xi,z)^{n+l}}dS(z),\ \xi\in\CO. $
Applying H\"{o}lder inequality twice we have
\begin{multline*}
 \abs{\intl_{\dom_s} g(\Psi(z)) \partial^\alpha f(z) dS(z)} \lesssim
 \intl_{\CO} \abs{\dbar \f(\xi)} \abs{ G_l(\xi)} d\mu(\xi)\\
 \lesssim \intl_\dom dS(\xi) \intl_0^\eps \abs{ \dbar \f(\xi_t) }
 \abs{G_l(\xi_t)} dt \\
 \lesssim \intl_\dom dS(\xi) \left(\intl_0^\eps \abs{\dbar \f(\xi_t)
 }^2 t^{-2l+1}dt\right)^{\frac12} \left(\intl_0^\eps \abs{ G_l(\xi_t)}^2 
 t^{2l-1} dt\right)^{\frac12}\\
 \lesssim \left( \intl_\dom dS(\xi) \left(\intl_0^\eps \abs{\dbar \f(\xi_t)t^{-l}}^2
 tdt\right)^{p/2}\right)^{1/p} \times\\ \times
 \left( \intl_\dom dS(\xi)  \left(\intl_0^\eps \abs{ G_l(\xi_t)}^2 t^{2l-1}
 dt\right)^{p'/2}\right)^{1/p'}.
\end{multline*}
The first product term is finite by~(\ref{ineq:PAC_Sobolev}) and
the second one is estimated by $\norm{g\circ\Psi }_{L^{p'}(\dom_s)}\lesssim \norm{g}_{L^{p'}(\dom)}$ in the view of area-integral inequality~(\ref{est:area_int}) in
Theorem~\ref{thm:area_int} because $\lambda(
\xi,z)$ is a polynomial with coefficient that continuously depend on $
\xi.$ \qed
\end{proof}

\section{Constructive description of Hardy-Sobolev spaces \label{Poly_Sobolev}}

\begin{theorem} \label{thm:Poly_Sobolev} Let $f\in H^1(\O)$ and $1<p<\infty,\
l\in\N.$ Then  $f\in H^l_p(\O)$ if and only if there exists a sequence of $2^k$-degree polynomials $P_{2^k}$ such that
\begin{equation} \label{ineq:Poly_Sobolev}
 \intl_{\dom} d\sigma(z) \left(\SK \abs{f(z)-P_{2^k}(z)}^2 2^{2l k}
 \right)^{p/2} < \infty.
\end{equation}
\end{theorem}

\begin{proof} This proof is a modification of the previous result for strictly convex domains obtained in \cite{R18_2}.
Assume that condition~(\ref{ineq:Poly_Sobolev}) holds, then
polynomials $P_{2^k}$ converge to the function $f$ in $L^p(\dom)$ and by
Theorem~\ref{thm:PAC_glob} we can construct pseudoanalytic continuation~$\f$ such that
\begin{equation*}
\abs{\dbar \f(z)} \lesssim \abs{P_{2^{k+1}}(z)-P_{2^k}(z)}
\rho(z)^{-1},\quad z\in \CO,\ 2^{-k}\leq \rho(z)<2^{-k+1}.
\end{equation*}
Let $t_k(z)\in[0,1]$ be such that $\rho(z+t_k(z))=2^{-k}.$ We define functions
\begin{align*}
a_k(z) &= \abs{P_{2^{k+1}}(z)-P_{2^k}(z)} 2^{kl },\\
b_k (z) &= \left(\intl_{t_k}^{t_{k-1}} \abs{\dbar
\f(z_t)t^{-l}}^2 tdt \right)^{1/2}, z\in\dom.
\end{align*}

\begin{lemma}\label{lm:Poly_Sobolev} $b_k(z) \lesssim M a_k (z),$
where $Ma_k$ is the maximal function on $\dom$ 
$$Ma_k(z) = \sup\limits_{r>0}\frac{1}{\abs{B(z,r)}} \int\limits_{B(z,r)} |a_k(\xi)| d\sigma(\xi).$$
\end{lemma}

Assume, that this lemma holds, then by the Fefferman-Stein maximal
theorem (see \cite{FS72,GLY04}) we have

$$ \intl_\dom \left(\SK b_k(z)^2\right)^{p/2} d\sigma(z) \lesssim  \intl_\dom \left(\SK a_k(z)^2\right)^{p/2}d\sigma(z).$$
The right-hand side of this inequality is finite by condition~(\ref{ineq:Poly_Sobolev}), also we have
$$ \SK b_k(z)^2 = \intl_{0}^\eps \abs{\dbar \f(z_t) t^{-l}}^2 tdt, $$
which completes the proof of the sufficiency in the theorem.

Let us prove the necessity. Now $f\in H^l_p(\O)$ with $1<p<\infty$ and
$l\in\N.$ By Theorem~\ref{thm:PAC_Sobolev} we could construct a
continuation~$\f$ of the function~$f$ with estimate~(\ref{ineq:PAC_Sobolev}). Applying the approximation of our kernel from Lemma~\ref{lm:KL_approx} to
the function $\f$ we define polynomials
\begin{equation*}
P_{2^k}(z) = \int_{\CO}
\bar{\partial} \f(\xi)\wedge K^{glob}_{2^k}(\xi,z).
\end{equation*} 
We will prove that these polynomials satisfy condition~(\ref{ineq:Poly_Sobolev}). 

From Lemma~\ref{lm:KL_approx} we obtain
\begin{multline*}
 \abs{f(z)-P_{2^k}(z)} \lesssim \int_{\CO} \abs{\bar{\partial}\f(\xi)} \abs{v(\xi,z)^{-n} - K^{glob}_{2^k}(\xi,z) } d\mu(\xi)\\ \lesssim U(z) + V(z) + W_1(z) +
 W_2(z),
\end{multline*}
where
\begin{align*}
U(z) &= \int\limits_{d(\xi,z)<t_k} \frac{\abs{\bar{\partial}\f(\xi)}}{ d(\xi,z)^{n}}d\mu(\xi),\quad V(z) = 2^{kn} \int\limits_{d(\xi,z)<t_k} \abs{ \bar{\partial}\f(\xi) } d\mu(\xi),  \\
W_1(z) &= \int\limits_{\substack{d(\xi,z)>t_k\\ \rho(\xi)<t_k}} \frac{ \abs{\bar{\partial}\f(\xi)} d\mu(\xi)}{2^{kr} d(\xi,z)^{n+r} },\quad
W_2(z) = \int\limits_{\rho(\xi)>t_k} \frac{
\abs{\bar{\partial}\f(\xi)} d\mu(\xi)}{
2^{kr} d(\xi,z)^{n+r} }.
\end{align*}
The parameter $r>0$ will be chosen later.

Note that $ V(z) \lesssim c U(z)$ and estimate the contribution of
 $U(z)$ to the sum. For some $c_1,c_2>0$ we have
\begin{multline*}
U(z) \leq \intl_{\substack{d(w,z)<c_1 t_k\\ w\in\dom}} d\sigma(w)
\sum\limits_{j>c_2 k} \intl_{t_j}^{t_{j-1}} \frac{ \abs{\dbar \f(w_t)} }{d(w_t,z)^n } dt\\
\leq \intl_{\substack{d(w,z)<c_1 t_k\\ w\in\dom}} d\sigma(w)
\sum\limits_{j>c_2 k} \left(\intl_{t_j}^{t_{j-1}} \abs{\dbar
\f(w_t)t^{-l}}^2 tdt\right)^{1/2}
\left(\intl_{t_j}^{t_{j-1}}  \frac{ t^{2l-1} dt }{ d(w_t,z)^n }\right)^{1/2}\\
= \sum\limits_{j>c_2 k} \intl_{d(w,z)<c_1 t_j} b_j(w) m_j(w)
d\sigma(w)
\end{multline*}
By Lemma~\ref{lm:QM2} we have $d(w_t,z) \asymp t + d(w,z)\gtrsim 2^{-j},$ $t\in[t_j,t_{j-1}].$ Hence, 
\begin{equation*}
 m_j(w)^2 = \intl_{t_j}^{t_{j-1}}  \frac{ t^{2(l-1)} dt }{ d(w_t,z)^n } \lesssim \frac{2^{-j(2l-1)}}{2^{-2jn}} 2^{-j} = 2^{-2j(l-n)}
\end{equation*}
and
\begin{equation} \label{est_U}
2^{kl} U(z) \lesssim \sum\limits_{j>c_1 k} 2^{-(j-k)l}
2^{jn}\intl_{d(w,z)< c_2 t_j } b_j(w) d\sigma(w)
 \lesssim  \sum\limits_{j>c_1 k} 2^{-(j-k)l} Mb_j(z).
\end{equation}
Now we estimate the value $W_1(z).$ Similarly to the previous we have
\begin{multline*}
W_1(z) \leq 2^{-kr} \suml_{j>k} \intl_{ d(w,z)\geq c_1 2^{-k} }
b_j(w) m_j^r(w) d\sigma(w)\\
\leq 2^{-kr} \suml_{j>k} \suml_{s=c_2}^k \intl_{ c_12^{-s}\leq d(w,z)\leq c_1 2^{-s+1} }
b_j(w) m_j^r(w) d\sigma(w),
\end{multline*}
where
\begin{equation*}
 m_j^r(w) =\left( \intl_{t_j}^{t_{j-1}} \frac{ t^{2l-1} dt }{ d(w_t,z)^{2(n+r)} } \right)^{1/2}.
\end{equation*}
Applying the estimate $d(w_t,z) \asymp t + d(w,z)\gtrsim
2^{-s},$ we obtain
\begin{equation*}
 m_j^r(w) \lesssim 2^{-jl+s(n+r)} .
\end{equation*}
Finally
\begin{equation*}
\suml_{t=c_2}^k \intl_{ d(w,z)\leq c_1 2^{-t+1} }
b_j(w) m_j^r(w) d\sigma(w) \lesssim \suml_{t=c_2}^k 2^{-jl+tr}Mb_j(z)\lesssim 2^{-jl+kr} Mb_j(z) 
\end{equation*}
and
\begin{equation} \label{est_W1}
 2^{kl} W_1(z) \lesssim \sum\limits_{j>k}  2^{-l(j-k)} Mb_j(z).
\end{equation}

Similarly, estimating the contribution of $W_2(z),$ we obtain

\begin{equation}\label{est_W2_pre}
2^{kl}W_2(z) \lesssim 2^{-k(r-l)} \sum\limits_{j=0}^k \intl_\dom
b_j(w) m_j^r(w) d\sigma(w).
\end{equation}
Since $d(w_t,z) \gtrsim 2^{-j}+d(w,z)$ for $w\in\dom,\ t\in[2^{-j},2^{-j+1}]$ then
$$m_j^r(w) \lesssim \frac{2^{-jl}}{(2^{-j} + d(w,z))^{n+r}}\leq \min\left(2^{j(n+r-l)}, 2^{-jl}d(w,z)^{-n-r}\right) .$$
Thus
\begin{equation*}
 \intl_\dom b_j(w) m_j^r(w) d\sigma(w) 
 \lesssim \sum\limits_{s=1}^{j} 2^{-jl} 2^{sr} M b_j(z) \lesssim
 2^{-jl} 2^{j r} M b_j(z).
\end{equation*}
Choosing $r=2l$ and applying estimate~(\ref{est_W2_pre})
\begin{equation} \label{est_W2}
 2^{kl} W_2(z)  \lesssim \sum\limits_{j=1}^k 2^{-(k-j)(r-l)} M
 b_j(z)\leq\sum\limits_{j=1}^k 2^{-(k-j)l} M
 b_j(z).
\end{equation}
Combining estimates~(\ref{est_U},~\ref{est_W1},~\ref{est_W2}) we
finally obtain
$$\abs{f(z)-P_{2^k}(z)}2^{kl} \lesssim \sum\limits_{j=1}^k 2^{-(k-j)l} M
 b_j(z) + \sum\limits_{j>k}  2^{-(j-k) l} M b_j(z), $$
which similarly to \cite{D81} implies
$$ \SK \abs{f(z)-P_{2^k}(z)}^2  2^{2kl} \lesssim \SK (M b_k(z))^2. $$
Then, by the Fefferman-Stein theorem (\cite{FS72}, \cite{GLY04})
\begin{multline*}\intl_{\dom} d\sigma(z) \left(\SK \abs{f(z)-P_{2^k}(z)}^2 2^{2l k}
 \right)^{p/2}\leq \intl_{\dom} \left(\SK b^2_k(z)\right)^{p/2}d\sigma(z)\\ \leq \intl_{\dom} d\sigma(z) \left( \intl_0^\eps \abs{\dbar \f(z_t) t^{-l}}^2 tdt \right)^{p/2}<\infty.
 \end{multline*}
This  completes the proof of the theorem and it remains to prove
Lemma~\ref{lm:Poly_Sobolev}. \qed
 \end{proof}

\begin{proof}[of Lemma~\ref{lm:Poly_Sobolev}] Define
$a_k(z) := 2^{kl}(P_{2^{k+1}}(z) - P_{2^k}(z)).$

Let $\xi\in\dom.$ A surface $V(\xi,z)=0$ defined by Levy polynomial (\ref{eq:LevyForm}) is near $\xi$ contained in $\C^n\setminus\O$ and can be locally defined by a holomorphic function $z:T_{\xi}\to\C^n$ such that $v(\xi,z(w) )=0$ near $\xi$ and $z(\xi)=\xi.$

Let $w\in T_{\xi}$ such that $\abs{w-\xi}<2^{-k/2}.$ Note that $z(w)\in\C^n\setminus\O$ and $\RE V(\xi,z(w))\geq 0.$ This implies that $z(w)+tn(\xi)\in\C^n\setminus\O$ and $\dist{z(w)+tn(\xi)}{\dom}\lesssim 2^{-k}.$ 

Consider
$$\tilde{\gamma}_{\xi,w}=\left\{u\in\C : z(w)+u n(\xi)\in\dom\right\}$$
and a closed curve $\gamma_{\xi,w}\subset\tilde{\gamma}_{\xi,w}$ that contains a point $p_{\xi,w},$ nearest to $0,$ and that bounds a simply connected domain $\O_{\xi,w}.$

 There exist a conformal map
$$\varphi_{\xi,w}:\C\setminus\O_{\xi,w} \to
\C\setminus\{v\in\C: \abs{v}\leq 1\}$$ such that $\
\varphi_{\xi,w}(\infty)=\infty,\ \varphi_{\xi,w}'(\infty)>0. $ 

We define an auxiliary function $G_k(u):=\frac{a_k(z(w)+u n(\xi))}{\varphi_{\xi,w}^{2^{k+1}}(u)}$ that is holomorphic in $\C\setminus\O_{\xi,w}.$ Applying the maximal estimate from~\cite{D77} to this function we have
\begin{multline*}
\abs{G_k(t)}\lesssim  \frac{1}{\dist{t}{\gamma_{\xi,w}}} \intl_{I_{\xi,w}}\abs{G_k(u)} \abs{du}\\ + \intl_{\gamma_{z,\tau}\setminus I_{\xi,w}} \abs{G_k(u)}\frac{\dist{t}{\gamma_{\xi,w}}^m}{\abs{u-t}^{m+1}} \abs{du} 
\end{multline*}
%\lesssim \frac{1}{t}\intl_{s\in I_{z,\tau}}\abs{G_k(s)} \abs{ds}+\norm{G_k}_{\gamma_{\xi,w}},
where $I_{\xi,w} = \{u\in\gamma_{\xi,w}: \abs{u-p_{\xi,w}}<\abs{p_{\xi,w}}/2\},$ and $m>0$ could be chosen arbitrary
large.

Note that $\abs{\varphi_{\xi,w}(t)}-1\asymp\dist{t}{\gamma_{\xi,w}}
\asymp\dist{z(w)+tn(\xi)}{\dom}\asymp 2^{-k}, $ thus $\abs{g_k(z(w)+tn(\xi))}\asymp
\abs{G_k(t)}$ for $2^{-k}\leq t\leq 2^{-k+1}.$ Hence,
\begin{equation} \label{lm:Poly_Sobolev_ineq1}
\abs{g_k(z(w)+tn(\xi))} \lesssim \suml_{j=1}^\infty 
\frac{1}{2^{-k+j(m+1)}} \intl_{\substack{u\in\gamma_{\xi,w}\\
\abs{u-p_{\xi,w}}<2^{j-k}}} \abs{g_k(u)} \abs{du}.
\end{equation}

Since the boundary of the domain $\O$ is $C^2$-smooth, we can assume that the constant in this inequality~(\ref{lm:Poly_Sobolev_ineq1})
does not depend on $\xi\in\dom$ and $w\in T_\xi.$

Note that the function $a_k(z(w)+tn(\xi))$ is holomorphic in $w\in T_\xi,$
then estimating the mean we obtain
\begin{multline*}
 \abs{a_k(\xi_t)}\leq \frac{1}{2^{-k(n-1)}} \intl_{\abs{w-\xi}<2^{-k/2}} \abs{a_k(z(w)+tn(\xi))} d\mu_{2n-2}(w)\\
 \lesssim \suml_{j=1}^\infty  \frac{1}{2^{-k(n-1)}} \intl_{ \abs{w-z}<2^{-k/2} } \frac{d\mu_{2n-2}(w) }{ 2^{-k+j(m+1)} } \intl_{\substack{u\in\gamma_{\xi,w}\\
\abs{u-p_{\xi,w}}<2^{j-k}}} \abs{a_k( z(w)+u n(\xi) )} \abs{du}\\ \lesssim \suml_{j=1}^\infty
 2^{-j(m-n+1)} \frac{1}{2^{(j-k)n}}\intl_{ B(\xi,2^{j-k}) } \abs{g_k}d\sigma,
\end{multline*}
where $d\mu_{2n-2}$ is the Lebesgue measure in $T_\xi.$

Assume that $m>n-1,$ then $ \abs{a_k(\xi_t)} \lesssim M a_k(\xi),\
\xi\in\dom,\ t\in[2^{-k},2^{-k+1}].$ Finally,
\begin{equation*}
b_k(\xi)^2 = \intl_{t_k}^{t_{k-1}} \abs{\dbar \f(\xi_t)t^{-l}}^2
tdt 
\lesssim \intl_{t_k}^{t_{k-1}} \abs{
a_k(\xi_t)}^2 \frac{dt}{t} \leq \left(Ma_k(\xi)\right)^{2}
\end{equation*}
and this completes the proof of the lemma.\qed
\end{proof}

\begin{remark} With small changes of proof we can state Theorem~\ref{thm:Poly_Sobolev} as follows. Let $f\in H(\O)$ and $1<p<\infty,\
l\in\N.$ Then  $f\in H^l_p(\O)$ if and only if there exists a sequence of $2^k$-degree polynomials $P_{2^k}$ such that
\begin{equation*} 
 \sup_{r<0}\intl_{\dom_r} d\sigma(z) \left(\SK \abs{f(z)-P_{2^k}(z)}^2 2^{2l k}
 \right)^{p/2} < \infty.
\end{equation*}
\end{remark}

\section{External g-function}

For function $g\in L^1(\dom)$ and $l\in\N$ we define \textit{the external g-function}
\begin{equation} \label{eq:area_int}
G_l(g,z) = \left(\ \intl_{0}^\eps \abs{\ \intl_\dom \frac{ g(u)
dS(u)}{ v(z_t,u)^{n+l} } }^2 t^{2l-1} dt\ \right)^{1/2},\ z\in\dom,
\end{equation}
where $dS(u)=\partial{\rho}\wedge(\bar\partial\partial{\rho})^{n-1}$ is a Leray-Levy measure.
\begin{theorem}\label{thm:area_int}
Let $\O$ be a strictly pseudoconvex Runge domain with $C^2-$smooth defining function and $g\in L^p(\dom),\ 1<p<\infty.$ Then
\begin{equation}\label{est:area_int}
    \intl_\dom G_l(g,z)^p dS(z)  \lesssim \intl_\dom \abs{g}^p
dS.
\end{equation}
\end{theorem}

\noindent Let $L^2=L^2([0,\eps],t^{2l-1}).$ We consider the function
\begin{equation} \label{eq:kernels}
 K(z,u) (t) = \frac{1}{v(z+tn(z),u)^{n+1}} = \frac{1}{v(z_t,u)^{n+1}}
\end{equation}
as a map $\dom\times\dom\to \mathscr{L}(\C,L^2)$ which values are operators of multiplication from $\C$ to
$L^2.$ Throughout the proof of the
Theorem~\ref{thm:area_int} $l$ will be fixed integer and the norm
of function $F$ in the space $L^2$ will be denoted by~$\norm{F}.$

We will show that integral operator $T$ defined by kernel $K$ is
bounded on $L^p(\dom)=L^p(\dom,dS).$ To prove this we apply $T1$-theorem for
transformations with operator-valued kernels formulated by
Hyt\"{o}nen and Weis in \cite{HW05}, taking in account that in our case concerned spaces are Hilbert. The proof of this theorem goes along the lines with the proof of the area-integral inequality for Cauchy-Leray-Fantappi\`{e} integral for strictly convex domain and complex ellipsoids from \cite{R18_1,R19}. However, the consideration of outer normal as a region of approach allows us to consider strictly pseudoconvex domain optimal in sense of smoothness.

Below we formulate the $T1$-theorem adapted to our context.

\begin{definition}
We say that the function $f\in C^\infty_0(\dom)$ is a normalized
bump-function associated with the quasiball $B(w,r)$ if
$\supp{f}\subset B(w,r),$ $\abs{f}\leq 1,$ and
$$\abs{f(\xi)-f(z)}\leq \frac{d(\xi,z)^{\gamma}}{r^{\gamma}},\ \xi,z\in\dom.$$
The set of bump-functions associated with $B(w,r)$ is denoted as
$A(\gamma,w,r).$
\end{definition}

\begin{theorem}[T1-theorem] \label{thm:T1}
Let $K:\dom\times\dom\to \mathscr{L} (\C, L^2)$ verify
the estimates
\begin{align}
    &\norm{K(z,u)} \lesssim \frac{1}{d(z,u)^n}; \label{KZ1}\\
    &\norm{K(z,u)-K(\xi,u)} \lesssim \frac{d(z,\xi)^\gamma}{d(z,u)^{n+\gamma}},\quad d(z,u)> C d(z,\xi); \label{KZ2}\\
    &\norm{K(z,u)-K(z,u')} \lesssim \frac{d(u,u')^\gamma}{d(z,u)^{n+\gamma}},\quad d(z,u)> C d(u,u')\label{KZ3}
\end{align}
for $\xi,z,u,u'\in\dom$ and some constants $C,\gamma>0.$

Assume that operator $T:
\mathscr{S}(\dom)\to\mathscr{S}'(\dom,\mathscr{L} (\C, L^2)$
with kernel $K$ verify the following conditions.
\begin{itemize}
 \item $T1,\ T'1\in\BMO(\dom,L^2),$ where $T'$ is a
 formally adjoint operator.
 \item Operator $T$ satisfies the weak boundedness property, that is,
 for every pair of normalized bump-functions $f,g\in A(\gamma,w,r)$ we have $$\norm{\left(g,Tf\right)}\lesssim
 r^{-n},$$
 where $(g,\ Tf)\in \mathscr{L}(\C,L^2)$ is a result of action of distribution $Tf\in\mathscr{S}'(\dom,\mathscr{L}(\C,L^2)$ on the function $g\in\mathscr{S}(\dom).$
\end{itemize}
Then $T\in\mathscr{L} (L^p(\dom), L^p(\dom,L^2))$ for
every $p\in (1,\infty).$
\end{theorem}
\medskip

In the following four lemmas we prove that kernel $K$ and
corresponding operator $T$ satisfy the conditions of 
$T1$-theorem with $\gamma=\frac12$. In particular, in lemmas~\ref{lm:T1_pre}, \ref{lm:T1_3} we prove that 
$T1,T'1\in L^\infty (\dom,L^2)\subset\BMO (\dom,L^2).$ 
%Through the proof we will use the notation
%\begin{equation}\label{eq:z_t}
%z_t=z+tn(z)
%\end{equation}

\begin{lemma} \label{lm:T1_1} The kernel $K$ verify estimates (\ref{KZ1}-\ref{KZ3}).

\end{lemma}
\begin{proof} By lemma \ref{lm:QM2} we have
$d(z_t,u)\asymp t + d(z,u),\ z,u\in\dom,\ t\in[0,\eps]. $
 Thus
\begin{equation*}
\norm{ K(z,u)}^2 = \intl_o^\eps \abs{K(z,u)(t)}^2 t^{2l-1}dt
 \lesssim \intl_0^\infty
\frac{t^{2l-1} dt}{(t+d(z,u))^{2n+2l}}\\ 
\lesssim \frac{1}{d(z,u)^{2n}}.
 \end{equation*}

To prove (\ref{KZ2}) we note that by Lemma \ref{lm:QM2} and due to the estimate $C d(\xi,z)\leq d(z,u)$ we have $d(z,\xi)\asymp d(\xi,u)$ and
$$d(z_t,u)\asymp t+d(z,\xi)\asymp t+d(\xi,u)\asymp d(\xi_t,u)$$
and that 
$$\abs{v(\xi_t,u)-v(z_t,u)}\lesssim d(\xi_t,z_t)+d(\xi_t,z_t)^{1/2}d(z_t,u)^{1/2} \lesssim d(\xi,z)^{1/2}d(z_t,u)^{1/2}.$$
Consequently,
\begin{multline*}
\norm{K(z,u) - K(\xi,u)}^2
\lesssim\intl_0^\eps \frac{d(\xi,z)(t+d(z,u))}{(t+d(z,u))^{2(n+l+1)}} t^{2l-1}dt\\
\leq \intl_0^\infty\frac{d(\xi,z)}{(t+d(z,u)^{2n+2}}dt
\lesssim \frac{d(\xi,z)}{d(z,u)^{2n+1}}.
\end{multline*}

Similarly, applying Lemma \ref{lm:QM2} 
%and the estimate $d(w,w')< C d(z,w)$ 
we obtain
\begin{equation*}
d(z_t,u)\asymp t +  d(z,u) \asymp t +  d(z,u') \asymp d(z_t,u')  
\end{equation*}
and
\begin{equation*}
 \abs{ v(z_t,u)-v(z_t,u') }  \lesssim d(z_t,u)^{1/2}d(u,u')^{1/2}.
\end{equation*}

Hence,
\begin{multline*}
\norm{ K(z,u) - K(z,u') }^2\lesssim
\intl_{0}^\eps \frac{ d(u,u') t^{2l-1} dt}{ d(z_t,u)^{2n+2l+1} }\\ 
\lesssim  \intl_0^\infty \frac{ d(u,u') t^{2l-1}
dt}{(t+d(z,u))^{2n+2l+1}} \lesssim \frac{d(u,u')}{d(z,u)^{n+1}}
\end{multline*}
and this finishes the proof of the lemma.\qed

\end{proof}

\begin{lemma}\label{lm:T1_pre} $\norm{T(1)} \lesssim 1.$
\end{lemma}

\begin{proof}
The function
$v(z_t,u)$ is holomorphic in $\O$ with respect to $w$ and by the Stokes theorem
\begin{equation*} %\label{eq:T1}
T(1)=\intl_\dom \frac{dS(u)}{v(z_t,u)^{n+l}} =  \intl_\O \frac{dV(u)}{v(z_t,u)^{n+l}}. 
\end{equation*}

By Lemma \ref{lm:QM2} we have 
$$d(z_t,u)
\asymp t + |\rho(u)| + d(z,\Psi(u)).$$
Hence
\begin{multline*} %\label{L2est1}
 \abs{T(1)(t)}\lesssim \intl_{\O}
\frac{d\mu(u)}{ \left(t-\rho(u)+d(z,\Psi(u))\right)^{n+l} }
\lesssim\intl_0^{T}
dt \intl_{\dom_s}\frac{d\sigma_s (u)}{\left(s+t+ d(z,\Psi(u) \right)^{n+l}}\\
\lesssim \intl_0^{T} dt \intl_0^\infty
\frac{v^{n-1}dv}{(t+s+v)^{n+l}}\lesssim \intl_0^{T}
\frac{dt}{(t+s)^l} \lesssim t^{1-l}
\ln{\left(1+\frac{1}{t}\right)},
\end{multline*}
and
\begin{equation*} %\label{L2est2}
\intl_0^\eps \abs{T(1)(t)}^2 t^{2l-1} \lesssim \intl_0^{\eps} \ln^2{\left(1+\frac{1}{t}\right)} t dt\lesssim 1.
\end{equation*}
This finishes the proof of the lemma.\qed
\end{proof}

\begin{lemma} \label{lm:T1_2} $\norm{T'(1)} \lesssim 1.$
\end{lemma}

\begin{proof} Consider
\begin{multline*}
 \overline{T'(1)}(u)(t)= \intl_\dom \frac{\chi(z) dS(z)}{v(z_t,u)^{n+l}} = -(n+1)\intl_{\O_{\eps}\setminus\O} \chi(z)\frac{\bar\partial_z \left(v(z_t,u)\right)\wedge dS(z) }{v(z_t,u)^{n+l+1}}\\
  + \intl_{\O_{\eps}\setminus\O} \frac{\chi(z) dV(z) }{v(z_t,u)^{n+l}} + \intl_{\O_{\eps}\setminus\O} \frac{\bar\partial_z\chi(z)\wedge dS(z) }{v(z_t,u)^{n+l}},
\end{multline*}
where $\chi(z)$ is some smooth function such that $\chi(z)=1,\ \rho(z)\leq\eps/2,$ and $\chi(z)=0,\ \rho(z)\geq\eps.$
Note that 
\begin{multline*}
 \abs{\bar\partial_{z_k} \left(v(z_t,u)\right)} \leq \sum\limits_{j=1}^n\abs{ (z_t-u)_j\bar\partial_{z_k} (w_j(z_t,u))} 
 + t\sum\limits_{j=1}^n \abs{w_j(z_t,u)\bar\partial_{z_k} n_j(z)}\\
 +\abs{\bar\partial_{z_k}\sum\limits_{k,j=1}^nP_{kj}(z_t,u)(z_t-u)_j(z_t-u)_k } \lesssim t +\abs{z-u}\lesssim\sqrt{t+d(z,u)}.
\end{multline*}
So we have
\begin{multline*}
\abs{T'(1)(u)(t)}
\lesssim \intl_{\O_{\eps}\setminus\O} \frac{dV(z) }{(t+\rho(z)+d(\Psi(z),u))^{n+l+1/2}}\\ \lesssim\intl_0^\eps ds\intl_{\dom_s} \frac{d\sigma_s(z)}{(t+s+d(\Psi(z),u))^{n+l+1/2}}\\
 \lesssim\intl_0^\eps ds\intl_0^{+\infty} \frac{ v^{n-1} dv}{(t+s+v)^{n+l+1/2}}
 \lesssim\intl_0^\eps \frac{ds}{(t+s)^{l+1/2}}\lesssim t^{-l+1/2}
\end{multline*}
and
$\norm{T'(1)(w)}\lesssim \intl_0^\eps t^{-2l+1}  t^{2l-1} dt\lesssim 1. $ \qed
\end{proof}

\begin{lemma} \label{lm:T1_3} Operator $T$ is weakly bounded.
\end{lemma}

\begin{proof}
 Let $f,g\in A\left(\frac{1}{2},u_0,r\right).$ Then
\begin{equation*}
\norm{\scp{g}{Tf}}^2 \lesssim \intl_{0}^\eps t^{2l-1} dt \left( \intl_{B(u_0,r)} \abs{g(z)} dS(z) \abs{\intl_{B(u_0,r)}\frac{f(w) dS(u)}{v(z_t,u)^{n+l}}
} \right)^2.
\end{equation*}

Denote $s(t):=\inf\limits_{u\in\dom} v(z_t,u) $ and
introduce the set
$$W(z,t,r) :=\left\{ w\in\dom: v(z_t,u)<s(t)+r\right\}.$$
Note that $\supp{f}\subset B(u_0,r)\subset W(z,t,cr) \subset B(z,c^2r)$ for some
$c>0,$ therefore,
\begin{multline*}
\abs{\intl_{B(u_0,r)}\frac{f(w) dS(u)}{v(z_t,u)^{n+l}} }
=\abs{\intl_{W(z,t,cr)}\frac{f(w)
dS(u)}{v(z_t,u)^{n+l}} }\\ \lesssim
\intl_{W(z,t,cr)}\frac{\abs{f(z)-f(w)} dS(u)}{
\abs{v(z_t,u)}^{n+l} } + \abs{f(z)} \abs{\intl_{\dom }\frac{dS(u)}{v(z_t,u)^{n+l}} }\\ +\abs{f(z)}\abs{\intl_{\dom\setminus W(z,t,cr) }\frac{dS(u)}{v(z_t,u)^{n+l}} } = L_1(z,t) +
\abs{f(z)} (L_2(z,t)+L_3(z,t)).
\end{multline*}

It follows from the estimate $|f(z)-f(w)|\leq \sqrt{d(w,z)/r}$ that
\begin{multline*}
L_1(z,t)\lesssim \frac{1}{\sqrt{r}}\intl_{B(z,c^2r)}
\frac{d(w,z)^{1/2}d\sigma(w)}{(t+ d(w,z))^{n+l}}\lesssim
\frac{1}{\sqrt{r}} \intl_0^{c^2r} \frac{v^{n-1/2}}{(t + v)^{n+l}}dv\\
\lesssim \frac{1}{\sqrt{r}} \intl_0^{c^2r}\frac{dv}{(t +
v)^{l+1/2}} \lesssim
\frac{1}{\sqrt{r}}\left(\frac{1}{t^{l-1/2}}-
\frac{1}{(t+c^2r)^{l-1/2}}\right)\\ = \frac{1}{\sqrt{r}}
\frac{(t+c^2 r)^{l-1/2} - t^{l-1/2}}{t^{l-1/2}
(t+r)^{l-1/2} } \lesssim \frac{1}{\sqrt{r}}
\frac{(t+c^2r)^{2l-1} - t^{2l-1} }{
t^{l-1/2}(t+c^2 r)^{2l-1}}\\ \lesssim
\frac{1}{\sqrt{r}} \frac{r t^{2l-2} +
r^{2l-1}}{t^{l-1/2}(t+r)^{2l-1}}\lesssim\frac{\sqrt{r}}{t^{l-1/2}(t+r)}.
\end{multline*}
Estimating the $L^2=L^2([0,\eps],t^{2l-1}dt)-$norm of the function $L_1(z,t)$ we obtain
\begin{equation}\label{est:I1}
\norm{L_1}^2=\intl_{0}^\eps L_1(z,t)^2 t^{2l-1}dt \lesssim \intl_0^\infty \frac{r}{(t+r)^2} dt\lesssim 1.
\end{equation} 

To estimate the second summand $L_2$ we apply the Stokes theorem to the domain
$$W_0=\left\{ u\in\O:
\abs{v(z_t,u)} > s(t)+c r\right\}$$ and to the form $\frac{dS(u)}{v(z_t,u)^{n+l}}$
\begin{multline}\label{est:I2}
\intl_{\dom\setminus W(z,t,cr)}\frac{dS(u)}{v(z_t,u)^{n+l}}
=\intl_{W_0}\frac{dV(u)}{v(z_t,u)^{n+l}} - \intl_{\substack{u\in\O\\ \abs{ v(z_t,u)} = s(t)+cr}
}\frac{dS(u)}{v(z_t,u)^{n+l}}\\ =L_4 -\frac{1}{(s(t)+cr)^{2n+2l}}
\intl_{\substack{u\in\O\\ d(z_t,u) =
t+cr}}\overline{v(z_t,u)}^{n+l} dS(u).
\end{multline}
By the proof of Lemma \ref{lm:T1_pre} $$\norm{L_4}\leq \intl_{W_0}\frac{dV(u)}{\abs{v(z_t,u)}^{n+l}}\leq \intl_{\O}\frac{dV(u)}{d(z_t,u)^{n+l}} \lesssim 1.$$ 

Applying Stokes' theorem to the domain 
$$\left\{ u\in\O: d(z_t,u)<s(t)+cr\right\}$$ 
we obtain
\begin{multline}\label{est:I3}
L_5:=\intl_{\substack{u\in\O\\
d(z_t,u)=t+cr}}\overline{v(z_t,u)}^{n+l} dS(u) =
-\intl_{\substack{u\in\dom\\ d(z_t,u)<t+cr}}\overline{v(z_t,u)}^{n+l} dS(u)\\
+ \intl_{\substack{u\in\O\\
d(z_t,u)<t+cr}}
\dbar_u\left(\overline{v(z_t,u)}^{n+l}\right)\wedge dS(u) +
\intl_{\substack{u\in\O\\
d(z_t,u)<t+cr}} \overline{v(z_t,u)}^{n+l} dV(u).
\end{multline}
Since 
\begin{equation}
\abs{\
\dbar_{u_j}\left(\overline{v(z_t,u)}^{n+l}\right)} = \abs{\
\dbar_{u_j}\left(\overline{\scp{w(z_t,u)}{z_t-u}}^{n+l}\right)} \lesssim d(z_t,u)^{n+l-1}
\end{equation}
and $s(t)\asymp t$ we get
\begin{multline}\label{est:I4} 
\abs{L_5} \lesssim \intl_{s(t)}^{s(t)+cr} (s^{n+l}s^{n-1} + s^{n+l-1}s^{n} +
s^{n+l}s^n) ds\\
\lesssim \intl_{s(t)}^{s(t)+cr} s^{2n+l-1} ds \lesssim r(t+r)^{2n+l-1}.
\end{multline}

Summarizing  the condition
$\abs{f(z)}\leq 1,\ z\in\dom,$ estimates (\ref{est:I1}-\ref{est:I4}) and Lemma \ref{lm:T1_pre}, we obtain
\begin{multline*}
\norm{\scp{g}{Tf}}^2\\
 \leq \intl_0^\eps dt \left(
\intl_{B(w_0,r)} \abs{g(z)} \left( L_1(z,t) + |f(z)|(L_2(z,t)+L_3(z,t))\right) dS(z)\right)^2 \\
\lesssim \norm{g}_{L^1(\dom)}^2
\sup\limits_{z\in\dom}\intl_0^\eps\left( L_1(z,t)^2 + L_2(z,t)^2+L_3(z,t)^2 \right)
dt\\ \lesssim \norm{g}_{L^1(\dom)}^2\lesssim
\abs{B(u_0,r)}^2.
\end{multline*}
The last estimate implies weak boundedness of operator $T$ and
completes the proof of the lemma. \qed
\end{proof}

\begin{proof}[of Theorem~\ref{thm:area_int}]
In Lemmas \ref{lm:T1_1}-\ref{lm:T1_3} we verified conditions of $T1-$theorem~\ref{thm:T1}. Consequently operator $T$ with kernel $F$ defined by (\ref{eq:kernels}) is bounded operator $L^p(\dom)\to L^p(\dom,L^2),$ where $L^2=L^2([0,\eps],t^{2l-1})$ and 
\begin{equation*}
    \intl_\dom G_l(g,z)^p dS(z) = \norm{Tg}^p_{L^p(\dom,L^2)}  \lesssim \intl_\dom \abs{g}^p
dS.
\end{equation*}
\end{proof}

\begin{remark} We note that the estimate (\ref{est:area_int}) in Theorem~\ref{thm:area_int}  holds for every $\Omega_t$ if $\abs{t}$ is small enough and the constant could be chosen independently on $t.$

\end{remark}

\end{document}